\documentclass[a4paper,11pt]{amsart}
\usepackage[english]{babel}
\usepackage{amssymb}
\usepackage{amscd}
\usepackage{mathtools}
\usepackage[paper=a4paper,left=3cm,right=3cm,top=2cm,bottom=2.0cm]{geometry}
\usepackage{enumerate}
\usepackage{hyperref}
\usepackage{tikz}
\usepackage{tikz-cd}

\newtheorem{thm}{Theorem}[section]
\newtheorem*{thm-non}{Theorem}

\newtheorem{lem}[thm]{Lemma}

\newtheorem{cor}[thm]{Corollary}
\theoremstyle{definition}

\newtheorem{rem}[thm]{Remark}

\DeclareMathOperator\PP{\mathbb{P}}
\DeclareMathOperator\OO{\mathcal{O}}
\DeclareMathOperator\HP{\mathbb{P}^{2[3]}}
\DeclareMathOperator\Db{D^b}

\DeclareMathOperator\Hom{Hom}
\DeclareMathOperator\Ext{Ext}
\DeclareMathOperator\ext{ext}
\DeclareMathOperator\Rex{\mathcal{E}xt^1_q}

\DeclareMathOperator\Bl{Bl}
\DeclareMathOperator\id{id}

\DeclareMathOperator\Ima{Im}
\DeclareMathOperator\cok{Coker}
\DeclareMathOperator\Pic{Pic}
\DeclareMathOperator\rk{rk}

\DeclareMathOperator\supp{supp}

\begin{document}

\title{The Fourier-Mukai transform of a universal family of stable vector bundles}

\author{Fabian Reede}
\address{Institut f\"ur Algebraische Geometrie, Leibniz Universit\"at Hannover, Welfengarten 1, 30167 Hannover, Germany}
\email{reede@math.uni-hannover.de}

\subjclass[2010]{14J60,14F05}

\begin{abstract}
In this note we prove that the Fourier-Mukai transform $\Phi_{\mathcal{U}}$ of the universal family of the moduli space $\mathcal{M}_{\PP^2}(4,1,3)$ is not fully faithful.
\end{abstract}

\maketitle

\section*{Introduction}
To every smooth projective variety $X$ one can associate its bounded derived category of coherent sheaves $\Db(X)$. The derived category contains a lot of geometric information about $X$. In some cases one can even recover $X$ from $\Db(X)$ but there are also examples of different varieties with equivalent derived categories, see \cite{huy2} for an introduction.

To compare the derived categories of two smooth projective varieties $X$ and $Y$, one needs to study functors between them. As it turns out, most of the interesting functors are Fourier-Mukai transforms $\Phi_{\mathcal{F}}: \Db(X) \rightarrow \Db(Y)$ for some object $\mathcal{F}\in \Db(X\times Y)$.

In this note we are interested in fully faithful Fourier-Mukai transforms because they give a semi-orthogonal decomposition of the derived category $\Db(Y)$ in smaller admissible subcategories. For example Krug and Sosna prove in \cite{krug3}  that the Fourier-Mukai transform $\Phi_{\mathcal{I}_{\mathcal{Z}}}: \Db(S) \rightarrow \Db(S^{[n]})$ induced by the universal ideal sheaf $\mathcal{I}_{\mathcal{Z}}$ of the Hilbert scheme $S^{[n]}$ is fully faithful for a surface $S$ with $p_g=q=0$, hence $\Db(S)$ is an admissible subcategory in $\Db(S^{[n]})$. This result was generalized for the Hilbert square $X^{[2]}$ to smooth projective varieties $X$ with exceptional structure sheaf and arbitrary dimension $\dim(X)\geq 2$, see \cite{bel}.

Another example of this behaviour is given by the moduli space $\mathcal{M}_C(2,L)$ of stable rank two vector bundles with fixed determinant $L$ of degree one on a smooth projective curve $C$ of genus $g\geq 2$. This moduli space is fine and thus there is a universal family $\mathcal{U}$ on $C\times \mathcal{M}_C(2,L)$. By work of Narasimhan, see \cite{nara1} and \cite{nara2}, as well as Fonarev and Kuznetsov, see \cite{fon}, it is known that the Fourier-Mukai transform $\Phi_{\mathcal{U}}: \Db(C)\rightarrow \Db(\mathcal{M}_C(2,L))$ is fully faithful. Thus $\Db(C)$ is an admissible subcategory of $\Db(\mathcal{M}_C(2,L))$. This also solves the so-called Fano visitor problem for smooth projective curves of genus $g\geq 2$. This result was generalized in \cite{bel2} to the higher rank case $\mathcal{M}_C(r,L)$ for a line bundle $L$ of degree $d$ such that $\text{gcd}(r,d)=1$ and curves of genus $g\geq g_0$ for some $g_0\in \mathbb{N}$.

In light of these examples one can ask if the Fourier-Mukai transform of the universal family $\mathcal{U}$ on a moduli space $\mathcal{M}_{\PP^2}(r,c_1,c_2)$ of stable sheaves on $\PP^2$ is also fully faithful. Our main result is, that this is not always the case. We prove:
\begin{thm-non}
	The Fourier-Mukai transform 
	\begin{equation*}
	\Phi_{\mathcal{U}}: \Db(\PP^2) \rightarrow \Db(\mathcal{M}_{\PP^2}(4,1,3))
	\end{equation*}
 induced by the universal family $\mathcal{U}$ of the moduli space $\mathcal{M}_{\PP^2}(4,1,3)$ is not fully faithful.
\end{thm-non}
The structure of this note is as follows: in section \ref{sect1} we recall some facts about the moduli space we are interested in. We construct an explicit family of stable sheaves for this moduli space in section \ref{sect2}. The computation of some cohomology groups for the family of stable sheaves can be found in section \ref{sect3}. In the final section \ref{sect4} we prove the main result.

Everything in this note is defined over the field of complex numbers $\mathbb{C}$. The projective plane $\PP^2$ is polarized by $H=\OO_{\PP^2}(1)$, thus $\mu$-stability means $\mu_{H}$-stability. A cohomology group written in lowercase characters simply denotes its dimension as a $\mathbb{C}$-vector space.

\subsection*{Acknowledgement}
I thank Pieter Belmans for asking me if it is possible to compute the $\Ext$-groups of the universal family and for pointing out some inaccuracies in an earlier draft of this note. I also thank Andreas Krug for explaining Lemma \ref{comp} to me.

\section{The moduli space}\label{sect1}
We begin by studying the moduli space $\mathcal{M}_{\PP^2}(4,1,3)$ of $S$-equivalence classes of $\mu$-semistable torsion-free sheaves $E$ on the projective plane $\PP^2$ with the following numerical data:
\begin{equation*}
\rk(E)=4,\,\,\,c_1(E)=1\,\,\, c_2(E)=3.
\end{equation*}
(Since the first Chern class is just an integer multiple of the polarization $H$, we simply identify it with this number.)

By this choice of rank $r$ and Chern classes $c_1$ resp. $c_2$ we get:
\begin{lem}
The moduli space $\mathcal{M}_{\PP^2}(4,1,3)$ is fine and there are no proper semistable sheaves.
\end{lem}
\begin{proof}
We have
\begin{equation*}
	\text{gcd}\left(r,c_1.H, \frac{1}{2}c_1.(c_1-K_{\PP^2})-c_2 \right)=\text{gcd}(4,1,-1)=1.
\end{equation*}
The result now follows from \cite[Corollary 4.6.7]{huy} and \cite[Remark 4.6.8]{huy}.
\end{proof}
\begin{rem}
This lemma shows that the moduli space $\mathcal{M}_{\PP^2}(4,1,3)$ has a universal family, that is a sheaf $\mathcal{U}$ on $\PP^2\times\mathcal{M}_{\PP^2}(4,1,3)$  flat over $\mathcal{M}_{\PP^2}(4,1,3)$ such that for every $E$ with $\left[E\right]\in \mathcal{M}_{\PP^2}(4,1,3)$ there is an isomorphism $	\mathcal{U}_{\left[E\right]}\cong E$, where $\mathcal{U}_{\left[E \right]}$ denotes the restriction of $\mathcal{U}$ to the fiber over $\left[E \right]$.
\end{rem}

The following properties of the moduli space are probably well known:
\begin{lem}\label{prop}
The moduli space $\mathcal{M}_{\PP^2}(4,1,3)$ is a smooth projective variety of dimension six. Furthermore all sheaves $E$ classified by this moduli space are locally free.
\end{lem}
\begin{proof}
The space $\mathcal{M}_{\PP^2}(4,1,3)$ is projective by construction. Since every sheaf $E$ is stable, we get by Serre duality
\begin{equation*}
	\Ext^2(E,E)\cong \Hom(E,E(-3))^{\vee} = 0
\end{equation*}
hence  $\mathcal{M}_{\PP^2}(4,1,3)$ is smooth and by \cite{elli} it is also irreducible. We also recall
\begin{equation*}
	\dim(\mathcal{M}_{\PP^2}(r,c_1,c_2))=\Delta-(r^2-1)\chi(\PP^2,\OO_{\PP^2})
\end{equation*}
where $\Delta=2rc_2-(r-1)c_1^2$ is the discriminant. So $\dim(\mathcal{M}_{\PP^2}(4,1,3))=6$.

The double dual of a $\mu$-stable torsion-free sheaf $E$ is still $\mu$-stable and defines a smooth point in $\mathcal{M}_{\PP^2}(4,1,3-\ell)$ with $\ell=\text{length}(E^{\vee\vee}/E)$. If $E$ were not locally free we would have $\ell\geq 1$ and $\mathcal{M}_{\PP^2}(4,1,3-\ell)$ would have negative dimension, which is not possible.
\end{proof}
\begin{rem}\label{locf}
Using Lemma \ref{prop} together with \cite[Lemma 2.1.7.]{huy} shows that the universal family $\mathcal{U}$ is itself locally free on $\PP^2\times\mathcal{M}_{\PP^2}(4,1,3)$. This implies that the sheaves $\mathcal{U}_p$, the restriction the fiber over $p\in \PP^2$, are also locally free on the moduli space.
\end{rem}

The sheaves classified by $\mathcal{M}_{\PP^2}(4,1,3)$ can be described more explicitly: 
\begin{lem}\label{exseq1}
	Let $E$ be a locally free sheaf on $\PP^2$ with $\left[E \right]\in \mathcal{M}_{\PP^2}(4,1,3)$, then there is a length three subscheme $Z\subset \PP^2$ and an exact sequence
\begin{equation*}
	\begin{tikzcd}
		0 \arrow[r] & \OO_{\PP^2}^{\oplus 3} \arrow[r] & E \arrow[r] & I_Z(1) \arrow[r] & 0.
	\end{tikzcd}
\end{equation*}	
\end{lem}

\begin{proof}
Hirzebruch-Riemann-Roch shows $\chi(\PP^2,E)=3$. The stability of $E$ implies that we have $h^2(\PP^2,E)=0$ and thus $h^0(\PP^2,E)\geq 3$.

Choose a 3-dimensional subspace $U\subset H^0(\PP^2,E)$, then by \cite[Lemma 1.5]{naka} the natural evaluation map $\varphi: U\otimes \OO_{\PP^2} \rightarrow E$ is injective with torsion-free quotient $Q=\cok(\varphi)$. We get the exact sequence
\begin{equation*}
	\begin{tikzcd}
		0 \arrow[r] & U\otimes\OO_{\PP^2} \arrow[r,"\varphi"] & E \arrow[r] & Q \arrow[r] & 0. 
	\end{tikzcd}
\end{equation*}	 
By comparing Chern classes we see that we must have $Q\cong I_Z(1)$ for a length three subscheme $Z\subset \PP^2$, that is $Z\in\HP$. This gives the desired exact sequence.
\end{proof}

Lemma \ref{exseq1} shows that there is a close connection between $\mathcal{M}_{\PP^2}(4,1,3)$ and $\HP$. This connection will become clearer in the next sections.

\section{Construction of a family}\label{sect2}
In this section we want to construct a $\HP$-family of $\mu$-stable locally free sheaves such that every member of this family is classified by $\mathcal{M}_{\PP^2}(4,1,3)$. The construction is based on a construction of Mukai, see \cite[Section 3]{muk}

The starting point of our construction is the observation that
\begin{equation}\label{obs}
	\ext^1(I_Z(1),\OO_{\PP^2})= h^1(\PP^2,I_Z(-2)) =3
\end{equation} 
for every $Z\in \HP$. 

We define $V:=\Ext^1(I_Z(1),\OO_{\PP^2})$ and observe the isomorphism 
\begin{equation*}
\Ext^1(I_Z(1),V^{\vee}\otimes\OO_{\PP^2})\cong\Ext^1(I_Z(1),\OO_{\PP^2})\otimes V^{\vee}\cong\Hom(V,V).
\end{equation*}
Hence there is a distinguished extension class $e\in \Ext^1(I_Z(1),V^{\vee}\otimes \OO_{\PP^2})$ corresponding to $\text{id}_V\in \Hom(V,V)$, giving rise to:
\begin{equation}\label{exseq2}
	\begin{tikzcd}
		0 \arrow{r} & V^{\vee}\otimes\OO_{\PP^2} \arrow{r} & E_Z \arrow{r} & I_Z(1) \arrow{r} & 0. 
	\end{tikzcd}
\end{equation}

\begin{rem}\label{hom}
	The sheaf $E_Z$ is called the universal extension of $I_Z(1)$ by $\OO_{\PP^2}$. By construction we have $\Hom(E_Z,\OO_{\PP^2})=0$.
\end{rem}

We want to study some of the properties of the sheaf $E_Z$. For example we have:
\begin{lem}
The sheaf $E_Z$  is a locally free sheaf on $\PP^2$.
\end{lem}

\begin{proof}
Tensor the exact sequence \eqref{exseq2} with $\omega_{\PP^2}$:
\begin{equation*}
	\begin{tikzcd}
		0 \arrow{r} & V^{\vee}\otimes\omega_{\PP^2} \arrow{r} & E_Z\otimes \omega_{\PP^2} \arrow{r} & I_Z(-2) \arrow{r} & 0 
	\end{tikzcd}.
\end{equation*}
Now for every subscheme $Z'\subsetneq Z$ of length $0\leq d <3$ we have
\begin{equation*}
	h^1(\PP^2,I_{Z'}(-2))<h^1(\PP^2,I_{Z}(-2)),
\end{equation*}
which by \cite[Lemma 1.2.]{tyu} implies that $E_Z\otimes \omega_{\PP^2}$ is locally free, hence so is $E_Z$. 
\end{proof}

We also have the following result concerning the stability of $E_Z$:
\begin{lem}
	The locally free sheaf $E_Z$ is $\mu$-stable.
\end{lem}
\begin{proof}
This follows from a more general result, see \cite[Lemma 1.4.]{naka}. But in this situation we can also give a direct proof:

Let $F$ be a torsion free quotient of $E_Z$ with $1\leq \rk(F)\leq 3$, then there is the following commutative diagram:
	\begin{equation*}
	\begin{tikzcd}
		0 \arrow{r} & \mathcal{O}_{\mathbb{P}^2}^{\oplus 3} \arrow{r}\arrow{d} & E_{Z} \arrow{r}\arrow{d} & I_{Z}(1) \arrow{r}\arrow{d} & 0\\
		0 \arrow{r} & F_0 \arrow{r} & F \arrow{r} & F_1 \arrow{r} & 0  
	\end{tikzcd}
\end{equation*}
with $F_0=\Ima(\mathcal{O}_{\mathbb{P}^2}^{\oplus 3}\hookrightarrow E_Z \rightarrow F)$. Thus all vertical arrows are surjective. Since $F_0$ is a quotient of a free sheaf we have $c_1(F_0).H\geq 0$. Furthermore $\rk(F_1)\in \left\lbrace0,1 \right\rbrace$ as $F_1$ is a quotient of a torsion free sheaf of rank 1. We distinguish two cases.

\underline{Case $\rk(F_1)=1$:}

In this case $F_1\cong I_Z(1)$ and hence $c_1(F).H=(c_1(F_0)+c_1(F_1)).H \geq 1$. This implies
\begin{equation*}
	\mu(F)=\frac{c_1(F).H}{\rk(F)}\geq \frac{1}{3} > \frac{1}{4} =\mu(E_Z).
\end{equation*}

\underline{Case $\rk(F_1)=0$:} 

In this case we have $c_1(F_1).H\geq 0$ as $F_1$ is a torsion sheaf. The only critical case is $c_1(F_0)=c_1(F_1)=0$, since otherwise $c_1(F)=d\geq 1$ and thus $\mu(F)\geq \frac{1}{3} > \frac{1}{4} =\mu(E_Z)$.

So assume $c_1(F_0)=c_1(F_1)=0$. Then $F_0$ is trivial itself, see for example \cite[p. 302]{laz}, and $F_1$ is supported in finitely many points. This implies 
\begin{equation*}
\Hom(F,\OO_{\PP^2})\cong \mathbb{C}^{\rk(F_0)}.
\end{equation*}
On the other hand $\Hom(F,\OO_{\PP^2})\hookrightarrow \Hom(E_Z,\OO_{\PP^2})=0$ by Remark \ref{hom}. This shows $\rk(F_0)=0$ and hence $\rk(F)=0$. So for $\rk(F)\geq 1$ the case $c_1(F_0)=c_1(F_1)=0$ cannot occur and $E_Z$ is stable.
\end{proof}
The last two lemmas show:
\begin{cor}\label{cor}
	For every $Z\in \HP$ the sheaf $E_Z$ defines a point $\left[E_Z \right]\in \mathcal{M}_{\PP^2}(4,1,3)$.  
\end{cor}

We want to put the $\mu$-stable locally free sheaves $E_Z$ in a family classified by $\HP$. To do this we need the following maps:
\begin{equation*}
	\begin{tikzcd}[column sep=small]
		\mathcal{Z}\arrow[hookrightarrow]{r} & \PP^2\times\HP \arrow[rightarrow]{dl}[swap]{p} \arrow[rightarrow]{dr}{q} & \\
		\PP^2  &&  \HP
	\end{tikzcd}
\end{equation*} 
where $\mathcal{Z}$ is the universal family of length 3 subschemes.

\begin{rem}
Recall that for any coherent sheaf $F$ on $\mathbb{P}^2$ there is the associated coherent \emph{tautological sheaf} $F^{[3]}$ on $\HP$ defined  by
 \begin{equation*}
 	F^{[3]}:=q_{*}\left(p^{*}F\otimes \OO_{\mathcal{Z}} \right). 
 \end{equation*}
If $F$ is locally free of rank $r$ then $F^{[3]}$ is locally free of rank $3r$.
\end{rem}
To construct the family of stable sheaves, we first put the $\Ext^1(I_Z(1),\OO_{\PP^2})$ for $Z\in \HP$ in a family:
\begin{lem}\label{rext}
	The first relative $\Ext$-sheaf $\mathcal{V}:=\Rex(\mathcal{I}_{\mathcal{Z}}\otimes p^{*}\OO_{\PP^2}(1),\OO_{\PP^2\times \HP})$ is a locally free sheaf of rank three on $\HP$. It commutes with base change and there is an isomorphism
	\begin{equation}\label{reliso}
		\Rex(\mathcal{I}_{\mathcal{Z}}\otimes p^{*}\OO_{\PP^2}(1),\OO_{\PP^2\times \HP})^{\vee} \cong \OO_{\PP^2}(-2)^{[3]}.
	\end{equation} 
\end{lem}

\begin{proof}
The morphism $q$ is proper and flat and the map
\begin{equation*}
	\phi: \HP \rightarrow \mathbb{N},\,\,\, Z \mapsto \ext^1(I_Z(1),\OO_{\PP^2})
\end{equation*}
is constant due to \eqref{obs}. So by \cite[Satz 3.]{schu} the first relative $\Ext$-sheaf
is locally free of rank three on $\HP$ and commutes with base change, that is for every $Z\in \HP$ we have
\begin{equation*}
	\Rex(\mathcal{I}_{\mathcal{Z}}\otimes p^{*}\OO_{\PP^2}(1),\OO_{\PP^2\times \HP})\otimes k(Z) \cong \Ext^1(I_Z(1),\OO_{\PP^2}).
\end{equation*}
	Using relative Serre duality, see \cite[Corollary(24)]{klei}, gives an isomorphism
	\begin{align*}
	\Rex(\mathcal{I}_{\mathcal{Z}}\otimes p^{*}\OO_{\PP^2}(1),\OO_{\PP^2\times \HP})&\cong \Rex(\mathcal{I}_{\mathcal{Z}}\otimes p^{*}\OO_{\PP^2}(-2),\omega_{q})\\
	&\cong \mathcal{H}om(R^1q_{*}(\mathcal{I}_{\mathcal{Z}}\otimes p^{*}\OO_{\PP^2}(-2)),\OO_{\HP}).
	\end{align*}
The exact sequence 
	\begin{equation*}
		\begin{tikzcd}
			0 \arrow{r} & \mathcal{I}_{\mathcal{Z}}\otimes p^{*}\OO_{\PP^2}(-2) \arrow{r} & p^{*}\OO_{\PP^2}(-2) \arrow{r} & \OO_{\mathcal{Z}}\otimes p^{*}\OO_{\PP^2}(-2) \arrow{r} & 0 
		\end{tikzcd}
	\end{equation*}
and standard cohomology and base change results, see \cite[II.5.]{mum}, show that there is an isomorphism
\begin{equation*}
 q_{*}(\OO_{\mathcal{Z}}\otimes p^{*}\OO_{\PP^2}(-2))\cong R^1q_{*}(\mathcal{I}_{\mathcal{Z}}\otimes p^{*}\OO_{\PP^2}(-2)). 
\end{equation*}
We see that $R^1q_{*}(\mathcal{I}_{\mathcal{Z}}\otimes p^{*}\OO_{\PP^2}(-2))\cong \OO_{\PP^2}(-2)^{[3]}$ is locally free of rank three and thus we get the desired isomorphism \eqref{reliso}.
\end{proof}

As the main result of this section we can now construct the desired family:

\begin{thm}\label{fam1}
There is a locally free $\HP$-family $\mathcal{E}$ of $\mu$-stable locally free sheaves, given by the exact sequence
\begin{equation*}
	\begin{tikzcd}
		0 \arrow{r} & q^{*}\mathcal{V}^{\vee}  \arrow{r} & \mathcal{E} \arrow{r} & \mathcal{I}_{\mathcal{Z}}\otimes p^{*}\OO_{\PP^2}(1) \arrow{r} & 0, 
	\end{tikzcd}
\end{equation*}
i.e. for every $Z\in \HP$ the restriction to the fiber over $Z$ defines a point $\left[\mathcal{E}_Z \right]\in \mathcal{M}_{\PP^2}(4,1,3)$.	
\end{thm}

\begin{proof}
For every $Z\in \HP$ we have $\Hom(I_Z(1),\OO_{\PP^2})=0$, so 
\begin{equation*}
	\mathcal{E}xt^0_q(\mathcal{I}_{\mathcal{Z}}\otimes p^{*}\OO_{\PP^2}(1),\OO_{\PP^2\times\HP})=q_{*}\mathcal{H}om(\mathcal{I}_{\mathcal{Z}}\otimes p^{*}\OO_{\PP^2}(1),\OO_{\PP^2\times\HP})=0.
\end{equation*}
Using this fact and the projection formula for relative $\Ext$-sheaves \cite[Lemma 4.1.]{lan}, the five term exact sequence of the spectral sequence
\begin{equation*}
	H^i(\HP,\mathcal{E}xt^j_q(\mathcal{I}_{\mathcal{Z}}\otimes p^{*}\OO_{\PP^2}(1),q^{*}\mathcal{V}^{\vee})) \Rightarrow \Ext^{i+j}(\mathcal{I}_{\mathcal{Z}}\otimes p^{*}\OO_{\PP^2}(1),q^{*}\mathcal{V}^{\vee})
\end{equation*}
reduces to an isomorphism
\begin{align*}
	\Ext^1(\mathcal{I}_{\mathcal{Z}}\otimes p^{*}\OO_{\PP^2}(1),q^{*}\mathcal{V}^{\vee})&\cong H^0(\HP,\mathcal{E}xt^1_q(\mathcal{I}_{\mathcal{Z}}\otimes p^{*}\OO_{\PP^2}(1),q^{*}\mathcal{V}^{\vee}))\\
	&\cong H^0(\HP,\mathcal{E}xt^1_q(\mathcal{I}_{\mathcal{Z}}\otimes p^{*}\OO_{\PP^2}(1),\OO_{\PP^2\times\HP})\otimes \mathcal{V}^{\vee})\\
	&\cong \Hom(\mathcal{V},\mathcal{V}).
\end{align*}
The identity $\id_{\mathcal{V}}$ gives rise to an extension on $\PP^2\times \HP$:
\begin{equation}\label{fam}
	\begin{tikzcd}
		0 \arrow{r} & q^{*}\mathcal{V}^{\vee}  \arrow{r} & \mathcal{E} \arrow{r} & \mathcal{I}_{\mathcal{Z}}\otimes p^{*}\OO_{\PP^2}(1) \arrow{r} & 0 
	\end{tikzcd}
\end{equation}
with $\mathcal{E}$ flat over $\HP$, since both other terms are. Restricting to the fiber over a point $Z\in \HP$ defines by flatness of $\mathcal{I}_{\mathcal{Z}}\otimes p^{*}\OO_{\PP^2}(1)$ a map
\begin{equation*}
	\Ext^1(\mathcal{I}_{\mathcal{Z}}\otimes p^{*}\OO_{\PP^2}(1),q^{*}\mathcal{V}^{\vee}) \rightarrow \Ext^1(I_Z(1),V^{\vee}\otimes\OO_{\PP^2}).
\end{equation*}
By \cite[Lemma 2.1.]{lan} the extension defined by $\id_{\mathcal{V}}$ restricts to the extension given by $\id_V$ on the fiber over $Z\in\HP$. Thus the pullback of \eqref{fam} to the fiber over $Z\in\HP$ is exactly the exact sequence \eqref{exseq2}, hence it defines a locally free sheaf classified by $\mathcal{M}_{\PP^2}(4,1,3)$. Using \cite[Lemma 2.1.7.]{huy} again, we see that $\mathcal{E}$ is itself locally free.	
\end{proof}

By the universal property of $\mathcal{M}_{\PP^2}(4,1,3)$ the family $\mathcal{E}$ comes with a classifying morphism 
\begin{equation*}
f_{\mathcal{E}}: \HP \rightarrow \mathcal{M}_{\PP^2}(4,1,3),\,\, Z \mapsto \left[\mathcal{E}_Z \right]. 
\end{equation*}
Furthermore there is $L\in \Pic(\HP)$ and an isomorphism
\begin{equation}\label{univ}
	(\id_{\PP^2}\times f_{\mathcal{E}})^{*}\mathcal{U}\otimes q^{*}L\cong \mathcal{E}.
\end{equation}

We need to study some properties of the morphism $f_{\mathcal{E}}$. For this we need:
\begin{lem}\label{iso}
	Assume $Z\in \HP$ is not collinear. If there is an isomorphism $\alpha: E_{Z'} \cong E_{Z}$ for some $Z'\in \HP$, then $Z=Z'$. 
\end{lem}
\begin{proof}
	We look at the following diagram:
	\begin{equation*}
		\begin{tikzcd}
			0 \arrow{r} & \mathcal{O}_{\mathbb{P}^2}^{\oplus 3} \arrow{r}{\iota} & E_{Z'} \arrow{r}\arrow{d}{\alpha}[swap]{\cong} & I_{Z'}(1) \arrow{r} & 0\\
			0 \arrow{r} & \mathcal{O}_{\mathbb{P}^2}^{\oplus 3} \arrow{r} & E_{Z} \arrow{r}{q} & I_{Z}(1) \arrow{r} & 0.
		\end{tikzcd}
	\end{equation*}
	Since $Z$ is not collinear the composition $\beta:=q\circ\alpha\circ\iota$ is zero. Consequently the free submodule of $E_{Z'}$ maps injectively to the free submodule of $E_Z$, which then must be an isomorphism, so we get in fact the following diagram:
	\begin{equation*}
		\begin{tikzcd}
			0 \arrow{r} & \mathcal{O}_{\mathbb{P}^2}^{\oplus 3} \arrow{r}{}\arrow{d}[swap]{\cong} & E_{Z'} \arrow{r}\arrow{d}{\alpha}[swap]{\cong} & I_{Z'}(1) \arrow{r}\arrow{d}[swap]{\cong} & 0\\
			0 \arrow{r} & \mathcal{O}_{\mathbb{P}^2}^{\oplus 3} \arrow{r} & E_{Z} \arrow{r}{} & I_{Z}(1) \arrow{r} & 0.
		\end{tikzcd}
	\end{equation*}
	Therefore there is an induced isomorphism $I_{Z'}(1)\cong I_Z(1)$ and so $Z=Z'$.
\end{proof}
Thus the non-collinear subschemes in $\HP$ define sheaves $E_Z$ with exactly three global sections. It makes sense to study the Brill-Noether-locus $S$ in $ \mathcal{M}_{\PP^2}(4,1,3)$:
\begin{equation*}
	S:=\left\lbrace \left[E \right]\in \mathcal{M}_{\PP^2}(4,1,3)\,|\, h^0(\PP^2,E)=4  \right\rbrace. 
\end{equation*} 
\begin{rem}
	We can write down the inverse $g$ to $f_{\mathcal{E}}$ on the complement of $S$:
	\begin{equation*}
		g: \mathcal{M}_{\PP^2}(4,1,3)\setminus S \rightarrow \HP,\,\,\, E \mapsto \supp(Q^{\vee\vee}/Q)
	\end{equation*}
where $Q$ is the cokernel of the (in this case) canonical evaluation map from Lemma \ref{exseq1}. 
\end{rem}
By \cite[Corollary, p.14, lines 3-5]{tyu} we get for the $E_Z$ with collinear subschemes $Z$:
\begin{lem}\label{niso}
	Assume $Z, Z'\in \HP$ are collinear with $Z\neq Z'$ such that there is a line $\ell\subset \PP^2$ containing both $Z$ and $Z'$, then $E_Z= E_{Z'}$.
\end{lem}
\begin{rem}
This shows that for a sheaf $E_Z$ with a collinear subscheme $Z\subset \PP^2$ we have
\begin{equation*}
	f_{\mathcal{E}}^{-1}(\left[ E_Z\right] )=\ell^{[3]}\cong \PP^3,
\end{equation*}
where $\ell\subset \PP^2$ is the line containing $Z$.
\end{rem}
The last two lemmas suggest that $f_{\mathcal{E}}$ is the blow up of $S$ in $\mathcal{M}_{\PP^2}(4,1,3)$. This is indeed the case since by \cite[5.29, Example 5.3.]{yosh} and we have:
\begin{lem}\label{bir}
	The Brill-Noether-locus $S$ in $\mathcal{M}_{\PP^2}(4,1,3)$ is isomorphic to $\PP^2$ and  there is an isomorphism $\HP \cong \Bl_S\mathcal{M}_{\PP^2}(4,1,3)$ such that $f_{\mathcal{E}}$ can be identified with the blow up of $S$ in $\mathcal{M}_{\PP^2}(4,1,3)$.
\end{lem}

\begin{rem}
This description goes back to Drezet who proved this in terms of Kronecker modules in \cite[Th\'eor\`eme 4.]{dre}.	
\end{rem}

\begin{cor}\label{coh}
For every locally free sheaf $F$ on $\mathcal{M}_{\PP^2}(4,1,3)$ we have isomorphisms
\begin{equation*}
H^i(\mathcal{M}_{\PP^2}(4,1,3),F)\cong H^i(\HP, f_{\mathcal{E}}^{*}F).
\end{equation*}
\end{cor}
\begin{proof}
Since $f_{\mathcal{E}}$ is birational by Lemma \ref{bir}, the result follows from $R^i\left(f_\mathcal{E} \right) _{*}\OO_{\HP}=0$ for $i\geq 1$, the projection formula and the Leray spectral sequence.
\end{proof}

\section{Computations}\label{sect3}
We want to understand the family $\mathcal{E}$ as a $\PP^2$-family, that is we want to understand the sheaves $\mathcal{E}_p$ on $\HP$ for $p\in \PP^2$. For this we first note that $\mathcal{Z}$ is not just flat over $\HP$ but also over $\PP^2$, see \cite[Theorem 2.1.]{krug1}, so restricting the exact sequence 
\begin{equation*}
	\begin{tikzcd}
		0 \arrow{r} & \mathcal{I}_{\mathcal{Z}}  \arrow{r} & \OO_{\PP^2\times\HP} \arrow{r} &  \OO_{\mathcal{Z}}\arrow{r} & 0 
	\end{tikzcd}
\end{equation*}
to the fiber over  point $p\in \PP^2$ gives the exact sequence
\begin{equation}\label{ide}
	\begin{tikzcd}
		0 \arrow{r} & I_{S_p}  \arrow{r} & \OO_{\HP} \arrow{r} &  \OO_{S_p} \arrow{r} & 0,
	\end{tikzcd}
\end{equation}
where $S_p:=\left\lbrace Z\in \HP\,|\, p\in \supp(Z) \right\rbrace$ is a codimension two subscheme in $\HP$. 

Since $\mathcal{Z}$ is flat over $\PP^2$ we see, using \cite[Examples 5.4 vi)]{huy2}, that $\OO_{S_p}=k(p)^{[3]}$ is the tautological sheaf on $\HP$ associated to the skyscraper sheaf $k(p)$ of the point $p\in \PP^2$. This implies we can use \cite[Theorem 3.17.,Remark 3.20.]{krug3} to find the following cohomology groups:
\begin{equation}\label{krug1}
	\ext^i(\OO_{\HP},\OO_{\PP^2}(-2)^{[3]})=0\,\,\text{for all $i$ and}\,\, \ext^i(\OO_{S_p},\OO_{\PP^2}(-2)^{[3]})=\begin{cases} 1 & i=2\\ 0 & i\neq 2\end{cases}
\end{equation}
as well as
\begin{equation}\label{krug2}
	\ext^i(\OO_{\PP^2}(-2)^{[3]},\OO_{\HP})=\begin{cases} 6 & i=0\\ 0 & i\geq 1\end{cases}\,\,\text{and}\,\, \ext^i(\OO_{\PP^2}(-2)^{[3]},\OO_{S_p})=\begin{cases} 7 & i=0\\ 0 & i\geq 1.\end{cases}
\end{equation}

Using these results we can prove:
\begin{lem}\label{comp}
	For $p\in \PP^2$ we have 
	\begin{equation*}
		\ext^i(I_{S_p},\OO_{\PP^2}(-2)^{[3]})=\begin{cases} 1 & i=1\\ 0 & i\neq 1\end{cases}\,\,\text{and}\,\,\ext^i(\OO_{\PP^2}(-2)^{[3]},I_{S_p})=\begin{cases} \geq 1 & i=1\\ 0 & i\geq 2.\end{cases}
	\end{equation*}
\end{lem}
\begin{proof}
Applying $\Hom(-,\OO_{\PP^2}(-2)^{[3]})$ to \eqref{ide} we have $\Ext^6(I_{S_p},\OO_{\PP^2}(-2)^{[3]})=0$ as well as isomorphisms:
\begin{equation*}
	\Ext^i(I_{S_p},\OO_{\PP^2}(-2)^{[3]})\cong \Ext^{i+1}(\OO_{S_p},\OO_{\PP^2}(-2)^{[3]})\,\,\,\text{for $1\leq i\leq 5$},
\end{equation*}
since $\Ext^i(\OO_{\HP},\OO_{\PP^2}(-2)^{[3]})=0$ for all $i$ by \eqref{krug1}.
We also find $\Hom(I_{S_p},\OO_{\PP^2}(-2)^{[3]})=0$ by further using $\Ext^i(\OO_{S_p},\OO_{\PP^2}(-2)^{[3]})=0$ for $i=0,1$. This proves the first claim.

For the second claim we apply $\Hom(\OO_{\PP^2}(-2)^{[3]},-)$ to \eqref{ide}. As $\Ext^1(\OO_{\PP^2}(-2)^{[3]},\OO_{\HP})=0$ by \eqref{krug2} the first part of the long exact sequence gives
\begin{equation*}
	\begin{tikzcd}
		 0 \arrow{r} & \Hom(\OO_{\PP^2}(-2)^{[3]},I_{S_p}) \arrow{r} & \mathbb{C}^6  \arrow{r} & \mathbb{C}^7 \arrow{r} &  \Ext^1(\OO_{\PP^2}(-2)^{[3]},I_{S_p}) \arrow{r} & 0 
	\end{tikzcd}
\end{equation*}
which shows that $\ext^1(\OO_{\PP^2}(-2)^{[3]},I_{S_p})\geq 1$.
We also get isomorphisms
\begin{equation*}
	\Ext^i(\OO_{\PP^2}(-2)^{[3]},I_{S_p})\cong \Ext^{i-1}(\OO_{\PP^2}(-2)^{[3]},\OO_{S_p})\,\,\,\text{ for $2\leq i \leq 6$}.
\end{equation*}
Again using \eqref{krug2} proves the second claim.
\end{proof}

To study the locally free sheaves $\mathcal{E}_p$ on $\HP$ we note that the exact sequence
\begin{equation*}
	\begin{tikzcd}
		0 \arrow{r} & q^{*}\mathcal{V}^{\vee}  \arrow{r} & \mathcal{E} \arrow{r} & \mathcal{I}_{\mathcal{Z}}\otimes p^{*}\OO_{\PP^2}(1) \arrow{r} & 0 
	\end{tikzcd}
\end{equation*}
restricts to the fiber over $p\in \PP^2$ as
\begin{equation}\label{ep}
	\begin{tikzcd}
		0 \arrow{r} & \OO_{\PP^2}(-2)^{[3]}  \arrow{r} & \mathcal{E}_p \arrow{r} & I_{S_p} \arrow{r} & 0 
	\end{tikzcd}
\end{equation}
by using flatness of $\mathcal{I}_{\mathcal{Z}}$ over $\PP^2$ and  Lemma \ref{rext}. We can now prove: 

\begin{thm}\label{ext1}
	Let $\mathcal{E}$ be the $\HP$-family of $\mu$-stable locally free sheaves, then for any pair of closed points $p,q\in \PP^2$ with $p\neq q$ we have
	\begin{equation*}
		\ext^1(\mathcal{E}_p,\mathcal{E}_q)\geq 1.
	\end{equation*}
\end{thm}

\begin{proof}
By \cite[Theorem 1.2]{krug2} the Fourier-Mukai transform
\begin{equation*}
\Phi_{\mathcal{I}_{\mathcal{Z}}}: \Db(\PP^2) \rightarrow \Db(\HP)
\end{equation*}
is fully faithful, that is for $p,q\in \PP^2$ with $p\neq q$  we have by flatness of $\mathcal{I}_{\mathcal{Z}}$ over $\PP^2$:
\begin{equation*}
	\Ext^i(I_{S_p},I_{S_q})\cong \Ext^i(k(p),k(q))= 0\,\,\,\text{for $0\leq i\leq 6$}.
\end{equation*}
So applying $\Hom(I_{S_q},-)$ with $q\neq p$ to \eqref{ep} gives isomorphisms 
\begin{equation*}
	\Ext^i(I_{S_q},\mathcal{E}_p)\cong \Ext^i(I_{S_q},\OO_{\PP^2}(-2)^{[3]})\,\,\,\text{for $0\leq i\leq 6$}.
\end{equation*}
If we apply $\Hom(\OO_{\PP^2}(-2)^{[3]},-)$ and use \cite[Theorem 3.17.]{krug3} again to see
\begin{equation*}
	\ext^i(\OO_{\PP^2}(-2)^{[3]},\OO_{\PP^2}(-2)^{[3]})=\begin{cases}
		1 & i=0 \\ 0 & i\geq 1
	\end{cases}
\end{equation*}
we get an exact sequence
\begin{equation*}
	\begin{tikzcd}
		0 \arrow{r} & \mathbb{C} \arrow{r} & \Hom(\OO_{\PP^2}(-2)^{[3]},\mathcal{E}_p) 
		\arrow{r} & \Hom(\OO_{\PP^2}(-2)^{[3]},I_{S_p})\arrow{r} & 0 
	\end{tikzcd}
\end{equation*}
and isomorphisms
\begin{equation*}
	\Ext^i(\OO_{\PP^2}(-2)^{[3]},\mathcal{E}_p)\cong 	\Ext^i(\OO_{\PP^2}(-2)^{[3]},I_{S_p})\,\,\,\text{for $1\leq i \leq 6$}.
\end{equation*}

Finally applying $\Hom(-,\mathcal{E}_q)$ with $q\neq p$ to \eqref{ep} we get the following relevant part of the induced long exact sequence:
\begin{equation*}
	\begin{tikzcd}
	\arrow{r} & \Ext^1(\mathcal{E}_p,\mathcal{E}_q) \arrow{r} & \Ext^1(\OO_{\PP^2}(-2)^{[3]},\mathcal{E}_q) 
		\arrow{r} & \Ext^2(I_{S_p},\mathcal{E}_q)\arrow{r} & {} 
	\end{tikzcd}
\end{equation*}
With the previous results this sequence gets:
\begin{equation*}
	\begin{tikzcd}
		 \arrow{r} &\Ext^1(\mathcal{E}_p,\mathcal{E}_q) \arrow{r} & \Ext^1(\OO_{\PP^2}(-2)^{[3]},I_{S_q}) \arrow{r} & \Ext^2(I_{S_p},\OO_{\PP^2}(-2)^{[3]}) \arrow{r} & {} 
	\end{tikzcd}
\end{equation*}
Using Lemma \ref{comp} we have $\Ext^2(I_{S_p},\OO_{\PP^2}(-2)^{[3]})=0$ and thus
\begin{equation*}
	\ext^1(\mathcal{E}_p,\mathcal{E}_q)\geq \ext^1(\OO_{\PP^2}(-2)^{[3]},I_{S_q})\geq 1.\qedhere
\end{equation*}
\end{proof}

\section{Non-full faithfulness of the universal family}\label{sect4}
We want to study the full faithfulness of the Fourier-Mukai transform 
\begin{equation*}
	\Phi_{\mathcal{U}}: \Db(\PP^2) \rightarrow \Db(\mathcal{M}_{\PP^2}(4,1,3))
\end{equation*}
induced by the universal family $\mathcal{U}$ of the moduli space $\mathcal{M}_{\PP^2}(4,1,3)$.

We will use the following corollary of the Bondal-Orlov criterion for full faithfulness:
\begin{lem}\cite[Corollary 7.5]{huy2}\label{orlo}
	Let $X$ and $Y$ be two smooth projective varieties and $\mathcal{P}$ a coherent sheaf on $X\times Y$, flat over $X$. Then the Fourier-Mukai transform 
	\begin{equation*}
		\Phi_{\mathcal{P}}: \Db(X) \rightarrow \Db(Y)
	\end{equation*}
is fully faithful if and only if the following two conditions are satisfied
\begin{enumerate}[i)]
	\item For any closed point $x\in X$ one has $\Ext^i(\mathcal{P}_x,\mathcal{P}_x)=\begin{cases}
		\mathbb{C} & i=0 \\ 0 & i>\dim(X)
	\end{cases}$
	\item For any pair of closed points $x,y\in X$ with $x\neq y$ one has $\Ext^i(\mathcal{P}_x,\mathcal{P}_y)=0$ for all i.
\end{enumerate}
\end{lem}

To apply this lemma to  $\mathcal{P}=\mathcal{U}$, the universal family of $\mathcal{M}_{\PP^2}(4,1,3)$, we need to be able to compute $\Ext^i(\mathcal{U}_p,\mathcal{U}_q)$. The following lemma reduces this problem to computing $\Ext^i(\mathcal{E}_p,\mathcal{E}_q)$:
\begin{lem}\label{exts}
	Let $\mathcal{U}$ be the universal family of $\mathcal{M}_{\PP^2}(4,1,3)$ and $\mathcal{E}$ be the $\HP$-family, then for any two points $p,q\in \PP^2$ there are the following isomorphisms for all $i$:
	\begin{equation*}
		\Ext^i(\mathcal{U}_p,\mathcal{U}_q)\cong \Ext^i(\mathcal{E}_p,\mathcal{E}_q).
	\end{equation*}
\end{lem}
\begin{proof}
We have the following chain of isomorphisms:
\begin{align*}
	\Ext^i(\mathcal{U}_p,\mathcal{U}_q) &\cong H^i(M,\mathcal{H}om(\mathcal{U}_p,\mathcal{U}_q))\\ 
	&\cong H^i(\HP,f_{\mathcal{E}}^{*}\mathcal{H}om(\mathcal{U}_p,\mathcal{U}_q))\\ 
	&\cong H^i(\HP,\mathcal{H}om(f_{\mathcal{E}}^{*}\mathcal{U}_p,f_{\mathcal{E}}^{*}\mathcal{U}_q))\\ 
	&\cong \Ext^i(f_{\mathcal{E}}^{*}\mathcal{U}_p,f_{\mathcal{E}}^{*}\mathcal{U}_q)\\ 
	&\cong \Ext^i(f_{\mathcal{E}}^{*}\mathcal{U}_p\otimes L,f_{\mathcal{E}}^{*}\mathcal{U}_q\otimes L)\\
	&\cong \Ext^i(\mathcal{E}_p,\mathcal{E}_q) 
\end{align*}	
Here the first and third isomorphism use the locally freeness of $\mathcal{U}_p$, see Remark \ref{locf}. The second isomorphism is Corollary \ref{coh}. The fourth isomorphism uses locally freeness of $f_{\mathcal{E}}^{*}\mathcal{U}_p$, while the sixth isomorphism follows from restricting \eqref{univ} to the fiber over $p\in \PP^2$.
\end{proof}
We can now prove the main theorem of this note:
\begin{thm}
	The Fourier-Mukai transform 
	\begin{equation*}
	\Phi_{\mathcal{U}}: \Db(\PP^2) \rightarrow \Db(\mathcal{M}_{\PP^2}(4,1,3))
	\end{equation*}
 induced by the universal family $\mathcal{U}$ of the moduli space $\mathcal{M}_{\PP^2}(4,1,3)$ is not fully faithful.
\end{thm}
\begin{proof}
For the Fourier-Mukai transform $\Phi_{\mathcal{U}}$ to be fully faithful one needs
\begin{equation}\label{van}
	\Ext^i(\mathcal{U}_p,\mathcal{U}_q)=0
\end{equation}
for any pair of points $p,q\in \PP^2$ with $p\neq q$ and any $i$ according to Lemma \ref{orlo}.

Lemma \ref{exts} shows that \eqref{van} is equivalent to
\begin{equation*}
	\Ext^i(\mathcal{E}_p,\mathcal{E}_q)=0.
\end{equation*}
But we have $\Ext^1(\mathcal{E}_p,\mathcal{E}_q)\neq 0$ by Lemma \ref{ext1}, so $\Phi_{\mathcal{U}}$ cannot be fully faithful.
\end{proof}

\end{document}